\documentclass[12pt]{amsart}

\usepackage{amsmath,amsthm,amssymb}

\usepackage{a4wide}
\usepackage{enumerate}

\newtheorem{theorem}{Theorem}[section]
\newtheorem{lemma}[theorem]{Lemma}
\newtheorem{proposition}[theorem]{Proposition}

\newtheorem*{maintheorem}{Main Theorem}

\theoremstyle{remark}
\newtheorem{remark}[theorem]{Remark}

\newcommand{\C}{\ensuremath{\mathbb{C}}}
\newcommand{\R}{\ensuremath{\mathbb{R}}}

\newcommand{\g}[1]{\ensuremath{\mathfrak{#1}}}
\newcommand{\II}{\ensuremath{I\! I}}

\DeclareMathOperator{\id}{id}

\DeclareMathOperator{\Exp}{Exp}

\DeclareMathOperator{\spann}{span}

\begin{document}
\title[Real hypersurfaces with two principal curvatures]{Real hypersurfaces with two principal curvatures in complex projective and hyperbolic planes}

\author[J.\ C.\ D\'{\i}az-Ramos]{Jos\'{e} Carlos D\'{\i}az-Ramos}
\author[M.\ Dom\'{\i}nguez-V\'{a}zquez]{Miguel Dom\'{\i}nguez-V\'{a}zquez}
\author[C.\ Vidal-Casti\~{n}eira]{Cristina Vidal-Casti\~{n}eira}

\address{Department of Geometry and Topology, University of Santiago de Compostela, Spain}
\email{josecarlos.diaz@usc.es}

\address{Department of Geometry and Topology, University of Santiago de Compostela, Spain}
\email{miguel.dominguez@usc.es}

\address{Departament de Matem\`atiques, Universitat Aut\`onoma de Barcelona, Spain}
\email{cvidal@mat.uab.cat}

\thanks{The authors have been supported by project MTM2009-07756 (Spain).}

\begin{abstract}
We find the first examples of real hypersurfaces with two nonconstant principal curvatures in complex projective and hyperbolic planes, and we classify them.

It turns out that each such hypersurface is foliated by equidistant Lagrangian flat surfaces with parallel mean curvature or, equivalently, by principal orbits of a cohomogeneity two polar action.
\end{abstract}




\maketitle


\section{Introduction and main result}
\label{sect:Intro}

The interest of studying real hypersurfaces in K\"ahler manifolds appeared in the field of Complex Analysis. In the theory of several complex variables, an important problem is to understand the relation between holomorphic functions defined on a domain of the complex space $\C^n$, and the boundary of such domain. When this boundary is smooth, it becomes a real hypersurface, that is, a submanifold of the Euclidean space $\R^{2n}$ with {real} codimension one. See~\cite{Ja95} for a survey on real hypersurfaces from the viewpoint of Complex Analysis.

From the point of view of Differential Geometry, a problem that has attracted the attention of many mathematicians over the last few decades is the classification of real hypersurfaces in terms of different geometric conditions. The case of real hypersurfaces in nonflat complex space forms
deserves special attention, for these spaces are the nonflat K\"ahler manifolds with the simplest curvature tensor.

The main reference for the study of real hypersurfaces in complex space forms is the influential survey \cite{NR97} by Niebergall and Ryan, where the authors reviewed the basic terminology and results in the field, and included a list of open problems which has motivated many investigations over the last years. One of the problems that was still outstanding, in spite of the efforts of several geometers, is Question 9.2 in \cite{NR97}:

\begin{quote}
\emph{Are there hypersurfaces in $\C P^2$ or $\C H^2$ that have two principal curvatures, other than the standard examples?}
\end{quote}

In this work, we give a surprising answer to this question. In order to explain the problem and our contribution, we start by introducing some notation
and terminology. Let $\bar{M}^n(c)$ denote a nonflat complex space form of complex dimension $n$, constant holomorphic curvature $c\neq 0$, and complex structure $J$. Hence, $\bar{M}^n(c)$ is isometric to a complex projective space $\C P^n$ of constant holomorphic curvature $c>0$, or to a complex hyperbolic space $\C H^n$ of constant holomorphic curvature $c<0$.

Let $M$ be a real hypersurface of $\bar{M}^n(c)$, that is, a submanifold of real codimension one, and $\xi$ a unit normal vector field along $M$. Then, $J\xi$ is tangent to $M$, and is called the {Hopf} vector field of $M$. We say that $M$ is {Hopf} at $p\in M$ if $J\xi$ is an eigenvector of the shape operator of $M$ at $p$. If $M$ is Hopf at every point, we say that $M$ is a {Hopf hypersurface}.

Tashiro and Tachibana~\cite{TT63} proved that there are no umbilical real hypersurfaces (i.e.\ real hypersurfaces with exactly one principal curvature) in nonflat complex space forms. Cecil and Ryan~\cite{CR82} showed that a real hypersurface with two principal curvatures in $\C P^n$, $n\geq 3$, has constant principal curvatures, and is an open part of a geodesic sphere. Montiel~\cite{Mo85} obtained an analogous result for $\C H^n$, $n\geq 3$, showing that real hypersurfaces with two distinct principal curvatures must be open parts of geodesic spheres, tubes around a totally geodesic $\C H^{n-1}$ in $\C H^n$, tubes of radius $\frac{1}{\sqrt{-c}}\log(2+\sqrt{3})$ around a totally geodesic $\R H^n$ in $\C H^n$, or horospheres. In both cases, the examples are open parts of Hopf hypersurfaces that are homogeneous, that is, orbits of isometric actions on $\bar{M}^n(c)$.

The methods used by Cecil and Ryan, and Montiel, do not work if $n=2$. The question above states the interest of extending the classification of real hypersurfaces with two distinct principal curvatures to $\C P^2$ and $\C H^2$, and poses the problem of the existence of examples with nonconstant principal curvatures. In this work we answer Question~9.2 in \cite{NR97} affirmatively for $\C P^2$ and $\C H^2$, and obtain a complete description of all the examples.

The construction of the new hypersurfaces makes use of the notion of polar action, which we briefly recall now; see~\cite[\S3.2]{BCO03},~\cite{DR13}, and~\cite{DDK} for more information. A proper isometric action of a group $H$ on $\bar{M}^2(c)$ is called {polar} if it admits a {section}, i.e.\ a totally geodesic submanifold $\Sigma\subset \bar{M}^2(c)$ which intersects all the orbits of the action and always orthogonally. The cohomogeneity of the action is the lowest codimension of the orbits. A point $p$ in $\Sigma$ is called {regular} if the orbit through $p$ is of maximal dimension, that is, its codimension equals the cohomogeneity of the action. It is known that the subset of regular points in $\Sigma$ is an open and dense subset of $\Sigma$. Moreover, the section $\Sigma$ of a polar action of cohomogeneity two is a totally geodesic real projective plane $\R P^2$ if $c>0$, or a totally geodesic real hyperbolic plane $\R H^2$ if $c<0$.

We state now the main contribution of this work.

\begin{maintheorem}
Let $\bar{M}^2(c)$ be a complex space form of complex dimension $2$ and constant holomorphic curvature $c\neq 0$. Consider a polar action of a group $H$ acting with cohomogeneity two and with section $\Sigma$ on $\bar{M}^2(c)$.

Then, for any regular point $p\in \Sigma$ and any unit tangent vector $w\in T_p\Sigma$, there are exactly two different locally defined unit speed curves $\gamma_i\colon(-\epsilon,\epsilon)\to\Sigma$, $i=1,2$, with $\gamma_i(0)=p$ and $\dot{\gamma}_i(0)=w$, such that the set $H\cdot \gamma_i=\{h(\gamma_i(t)):h\in H,\, t\in(-\epsilon,\epsilon)\}$ is a real hypersurface with two principal curvatures in $\bar{M}^2(c)$. Generically, such hypersurface is non-Hopf and with nonconstant principal curvatures.

Conversely, let $M$ be an real hypersurface of $\bar{M}^2(c)$ with two nonconstant principal curvatures and which is non-Hopf at every point. Then $M$ is locally congruent to an open part of a real hypersurface constructed as above.
\end{maintheorem}

\begin{remark}\label{remark:ODE}
It is possible to construct the curves $\gamma_i$, $i=1,2$, in terms of a solution to an explicit system of ordinary differential equations, which we present now. Given $p\in \Sigma$ and $w\in T_p\Sigma$, let $\xi_p\in T_p\Sigma$ be a unit vector orthogonal to $w$. Let $\alpha_0$, $\beta_0$ be the principal curvatures of the surface $H\cdot p$ with respect to the normal vector $\xi_p$, $U_p$ the orthogonal projection of $J\xi_p$ onto the principal curvature space associated with $\alpha_0$, and $\phi_0\in[0,\pi/2]$ the angle between $J\xi_p$ and $U_p$. Let us assume that $\phi_0\in(0,\pi/2)$. By reversing the sign of $\xi_p$ if necessary, we can further assume that $\langle Jw,U_p\rangle>0$. Consider a local solution $(\tilde{\alpha},\tilde{\beta},\tilde{\phi})\colon (-\epsilon,\epsilon)\to \R^3$ to the ODE
\[
\begin{aligned}
\tilde{\alpha}' &=\frac{1}{4}\bigl(c(2-3\sin^2\tilde{\phi})
+4\tilde{\alpha}(\tilde{\alpha}-\tilde{\beta})\bigr)\tan\tilde{\phi},\\
\tilde{\beta}'
&=-\frac{3c}{8}\sin 2\tilde{\phi},\\
\tilde{\phi}'
&=\tilde{\beta}+\frac{c(1-3\sin^2\tilde{\phi})}{4(\tilde{\alpha}-\tilde{\beta})},
\end{aligned}
\]
with initial conditions $\tilde{\alpha}(0)=\alpha_0$, $\tilde{\beta}(0)=\beta_0$, and $\tilde{\phi}(0)=\phi_0$. Then, one of the two curves mentioned in the Main Theorem is the unique  curve $\gamma_\beta\colon(-\epsilon,\epsilon)\to\Sigma$ with curvature function $\tilde{\beta}$ and orientation given by $\xi_p$ (i.e.\ $\bar{\nabla}_{\dot{\gamma}_\beta(t)}\dot{\gamma}_\beta=\tilde{\beta}(t) \xi_{\gamma_\beta(t)}$ for all $t$, where $\bar{\nabla}$ is the Levi-Civita connection of $\Sigma$, and $\xi$ is a smooth normal vector field along $\gamma_\beta$ extending $\xi_p$ and tangent to $\Sigma$). The corresponding hypersurface $H\cdot\gamma_\beta$ has two principal curvatures $\alpha$ and $\beta$ (this one with multiplicity $2$) such that $(\alpha\circ\gamma_\beta)(t)=\tilde{\alpha}(t)$, and $(\beta\circ\gamma_\beta)(t)=\tilde{\beta}(t)$ for all $t$. The other possible curve $\gamma_\alpha$ is obtained by interchanging the roles of $\alpha_0$ and $\beta_0$ in the description above.
\end{remark}

Our Main Theorem guarantees that a real hypersurface under the above mentioned assumptions is foliated by orbits of maximal dimension of a polar action of cohomogeneity two on $\bar{M}^2(c)$. Polar actions on $\bar{M}^2(c)$ have been classified by Podest\`a and Thorbergsson~\cite{PT99} for $c>0$, and by Berndt and the first author~\cite{BD11a} for $c<0$. In particular, there is an explicit list of polar actions of cohomogeneity two on $\bar{M}^2(c)$. As an example, the group $H=U(1)\times U(1)\times U(1)$ acts polarly and with cohomogeneity two on both $\C P^2$ and $\C H^2$, and the corresponding principal orbits (equivalently, orbits of maximal dimension) are tori. This example is the only one, up to orbit equivalence, on $\C P^2$, but on $\C H^2$ there are three more polar actions up to orbit equivalence. See~Subsection~\ref{subsec:polar} for a description of these examples.

Interestingly, the proof of the Main Theorem relies on certain geometric characterization of the principal orbits of polar actions of cohomogeneity two on $\bar{M}^2(c)$, $c\neq 0$. As shown in Theorem~\ref{th:principal} below, such principal orbits are precisely the \emph{Lagrangian and flat surfaces with parallel mean curvature} in $\bar{M}^2(c)$. On the one hand, a Lagrangian surface is, in this context, a synonym of a totally real $2$-dimensional submanifold. On the other hand, the study of submanifolds with parallel mean curvature is an active field of research nowadays; see~\cite{Ch10} for a survey. In particular, the case of surfaces with parallel mean curvature in $2$-dimensional complex space forms deserves special attention, and is the subject of important recent advances, e.g.~\cite{Fe12}.

This work is organized as follows. In Section~\ref{sect:construction} we present the construction of the new examples of real hypersurfaces with two principal curvatures. Then we focus on the classification problem. In Section~\ref{sect:setup} we start by introducing some vector fields and functions naturally defined on the hypersurface. Then, in Section~\ref{sect:Levi-Civita}, we obtain an explicit expression for the Levi-Civita connection of the hypersurface. Using this information we show that our hypersurface is foliated orthogonally by equidistant, Lagrangian, flat surfaces with parallel second fundamental form, and by the integral curves of certain vector field (Section~\ref{sect:distributions}). Finally, in Section~\ref{sect:proof} we conclude the proof of the Main Theorem.

\textbf{Acknowledgments}. The authors would like to thank Professors\ J\"{u}rgen Berndt and Tillmann Jentsch for helpful observations and comments.


\section{Construction of the new examples}\label{sect:construction}

In this section we present a method to construct real hypersurfaces with two principal curvatures in nonflat complex space forms that generically have nonconstant principal curvatures.

We denote by $\bar{M}^2(c)$ a nonflat complex space form of complex dimension $2$ and constant holomorphic curvature $c\neq 0$. Hence, $\bar{M}^2(c)$ is isometric to a complex projective plane $\C P^2(c)$ of constant holomorphic curvature $c>0$, or a complex hyperbolic plane $\C H^2(c)$ of constant holomorphic curvature $c<0$. We denote by $\langle\,\cdot\,,\,\cdot\,\rangle$ the metric of $\bar{M}^2(c)$, by $J$ its complex structure, by $\bar{\nabla}$  its Levi-Civita connection, and by $\bar{R}$ its curvature tensor. Throughout this paper, we will adopt the following sign for the curvature tensor: $\bar{R}(X,Y)Z=[\bar{\nabla}_X,\bar{\nabla}_Y]Z-\bar{\nabla}_{[X,Y]}Z$. Recall also that the curvature tensor of a complex space form of constant holomorphic curvature $c$ is then written as
\begin{align*}
\langle\bar{R}(X,Y)V,W\rangle={}&
\frac{c}{4}\Bigl(
\langle Y,V\rangle\langle X,W\rangle
-\langle X,V\rangle\langle Y,W\rangle\\[-1ex]
&\phantom{\frac{c}{4}\Bigl(}
+\langle JY,V\rangle\langle JX,W\rangle
-\langle JX,V\rangle\langle JY,W\rangle
-2\langle JX,Y\rangle\langle JV,W\rangle
\Bigr).
\end{align*}

If $N$ is a submanifold of $\bar{M}^2(c)$, we will denote by $\II$ its second fundamental form, which is a symmetric tensor defined as $\II(X,Y)=(\bar{\nabla}_XY)^\perp$, where $X$ and $Y$ are tangent vectors to $N$, and $(\cdot)^\perp$ denotes orthogonal projection onto the normal bundle $\nu N$. If $\xi$ is a vector perpendicular to $N$, then we define the shape operator of $N$ with respect to $\xi$ by the formula $\langle S_\xi X,Y\rangle=\langle\II(X,Y),\xi\rangle$.\medskip

The idea behind our construction is rather simple. We start with a polar action of a group $H$ acting with cohomogeneity two on $\bar{M}^2(c)$. In one of its $2$-dimensional sections we find a (locally defined) curve $\gamma$ such that, if we attach to each point $\gamma(t)$ of $\gamma$ the $H$-orbit through $\gamma(t)$, we obtain a $3$-dimensional real hypersurface $M$ of $\bar{M}^2(c)$ with two principal curvatures at every point. The fact that $M$ has two principal curvatures is equivalent to the fact that $\gamma$ satisfies certain ordinary differential equation. Using the theorem of existence of solutions to ODEs in canonical form, we derive the existence of our hypersurfaces.

\subsection{Polar actions and Lagrangian flat surfaces with parallel mean curvature}\label{subsec:polar}

An action $H\times M\to M$, $(h,p)\mapsto h(p)$, of a group $H$ on a manifold $M$ is said to be isometric if for each $h\in H$ the map $p\mapsto h(p)$ is an isometry of $M$. It is customary to take $H$ a closed subgroup of the isometry group of $M$ so that the action is effective and proper; thus isotropy groups are compact, orbits are closed and embedded, and the orbit space is Hausdorff. Two isometric actions are said to be \emph{orbit equivalent} if there exists an isometry of $M$ that maps the orbits of one action to the orbits of the other action. A principal orbit is an orbit whose isotropy groups are minimal (with respect to the inclusion relation) up to conjugation in $H$. It turns out that principal orbits have maximal dimension. By definition, the codimension of a principal orbit is called the \emph{cohomogeneity} of the action. The union of all principal orbits is an open dense subset of $M$. We will say that a point $p\in M$ is \emph{regular} if $p$ is contained in a principal orbit.

An isometric action is said to be \emph{polar} if it has a \emph{section}, that is, a submanifold $\Sigma$ of $M$ that intersects all the orbits of $H$, and whose tangent space at any point $p\in\Sigma$ is perpendicular to the tangent space of the orbit of $H$ through $p$. A polar action admits a section through any given point. It follows from~\cite[Corollary 1.3]{Ly10} that an orbit of a polar action on a complex space form is principal if and only if has maximum dimension.

Polar actions on $2$-dimensional nonflat complex space forms have been classified in~\cite{PT99} and~\cite{BD11a}. Polar actions on $\C P^2$ or $\C H^2$ are either of cohomogeneity one or cohomogeneity two. It follows from this classification that, in the cohomogeneity two case, a section~$\Sigma$ is a totally geodesic real projective plane $\R P^2$ in $\C P^2$ if $c>0$, or a totally geodesic real hyperbolic plane $\R H^2$ in $\C H^2$ if $c<0$. We briefly discuss these classifications below. Let $H$ be a connected closed subgroup of isometries of $\bar{M}^2(c)$, $c\neq 0$, acting polarly on $\bar{M}^2(c)$ with cohomogeneity two.

In $\C P^2$ there is only one polar action of cohomogeneity two up to orbit equivalence, which corresponds to the action of the group $H=U(1)\times U(1)\times U(1)$ on $\C P^2$. The principal orbits of $H$ are $2$-dimensional tori in $\C P^2$.

In $\C H^2$ there are exactly four examples up to orbit equivalence. In order to describe them we have to introduce some notation (see~\cite{BD11a} for a detailed exposition). Let $\g{g}=\g{su}(1,2)$ be the Lie algebra of the isometry group of $\C H^2$, and $\g{k}=\g{s}(\g{u}(1)\oplus\g{u}(2))$ the isotropy Lie algebra at some point of $\C H^2$. Let $\g{g}=\g{k}\oplus\g{p}$ be the corresponding Cartan decomposition, where $\g{p}$ is the orthogonal complement of $\g{k}$ in $\g{g}$ with respect to the Killing form of $\g{g}$. Let $\g{a}$ be a maximal abelian subspace of $\g{p}$, and $\g{g}=\g{g}_{-2\alpha}\oplus\g{g}_{-\alpha}\oplus\g{g}_{0}
\oplus\g{g}_{\alpha}\oplus\g{g}_{2\alpha}$ the corresponding restricted root space decomposition. The root space $\g{g}_0$ decomposes as $\g{g}_0=\g{k}_0\oplus\g{a}$, where $\g{k}_0\cong\g{u}(1)$ is the centralizer of $\g{a}$ in $\g{k}$.

The subgroup $H=U(1)\times U(1)\times U(1)$ of $U(1,2)$ acts polarly  on $\C H^2$ with cohomogeneity two, and its principal orbits are $2$-dimensional tori in $\C H^2$. The other three cohomogeneity two polar actions on $\C H^2$ correspond to the action of the connected subgroups $H$ of $SU(1,2)$ with one of the following Lie algebras: $\g{h}=\g{g}_0$, $\g{h}=\g{k}_0\oplus\g{g}_{2\alpha}$, or $\g{h}=\ell\oplus\g{g}_{2\alpha}$, where~$\ell$ is a $1$-dimensional linear subspace of $\g{g}_\alpha$. Topologically, the principal orbits of these first two actions are $2$-dimensional cylinders, while those of the last one are flat $2$-planes.

The complex structure $J$ of the complex space form $\bar{M}^2(c)$ defines a symplectic structure $\Omega$ on $\bar{M}^2(c)$ by $\Omega(X,Y)=\langle JX,Y\rangle$, with $X$ and $Y$ vector fields on $\bar{M}^2(c)$. A submanifold $N$ is said to be \emph{Lagrangian} if it is a submanifold of maximum possible dimension such that $\Omega$ vanishes identically on each tangent space of $N$. Recall also that a submanifold $N$ is said to be \emph{totally real} if $JT_pN$ is perpendicular to $T_pN$ for any $p\in N$. It is easy to see that in $\bar{M}^2(c)$, $2$-dimensional submanifolds are Lagrangian precisely if they are totally real.

The following result provides a geometric characterization of the principal orbits of polar actions of cohomogeneity two on $\C P^2$ and $\C H^2$. It will be important for the construction of our examples, and crucial for the proof of the Main Theorem.

\begin{theorem}\label{th:principal}
Let $N$ be a submanifold of $\bar{M}^2(c)$, $c\neq 0$. Then the following statements are equivalent:
\begin{enumerate}[\rm (i)]
\item $N$ is an open part of a principal orbit of a cohomogeneity two polar action on $\bar{M}^2(c)$.
\item $N$ is a Lagrangian, flat surface of $\bar{M}^2(c)$ with parallel mean curvature.
\end{enumerate}
In this situation, $N$ has parallel second fundamental form and flat normal bundle.
\end{theorem}

\begin{proof}
First, let $H$ be a connected group of isometries of $\bar{M}^2(c)$, $c\neq 0$, acting polarly and with cohomogeneity two on $\bar{M}^2(c)$, and let $H\cdot p$ be a principal orbit through some regular point $p$. We can assume that $N=H\cdot p$. Given a section $\Sigma$ through $p$, we have the orthogonal direct sum $T_p\bar{M}^2(c)=T_p\Sigma\oplus T_pN$. Since the sections of $H$ are totally real submanifolds of $\bar{M}^2(c)$, it follows that $N$ is a Lagrangian submanifold. Furthermore, it is well known (see~\cite[Corollary~3.2.5]{BCO03}) that principal orbits of polar actions have flat normal bundle. In fact, every $H$-equivariant normal vector field along $N$ is parallel with respect to the normal connection $\nabla^\perp$ of $N$. Thus, let $\{W_1,W_2\}$ be an orthonormal frame of $H$-equivariant (i.e.\ $W_i=h_*W_i$ for all $h\in H$, $i=1,2$) normal vector fields along $N$. Define an orthonormal frame of tangent vector fields along $N$ by $U_i=-JW_i$, $i=1,2$. Since $H$ is connected, the isometries in $H$ are holomorphic. Hence $U_i=h_*U_i$ for all $h\in H$, $i=1,2$. Moreover, if we denote by $\bar{\nabla}$ and $\nabla$  the Levi-Civita connections of $\bar{M}^2(c)$ and $N$ respectively, and by $\II$ the second fundamental form of $N$, we have:
\[
0=\nabla^\perp_X W_i=(\bar{\nabla}_X JU_i)^\perp=(J\bar{\nabla}_X U_i)^\perp=(J(\nabla_X U_i+ \II(X,U_i)))^\perp=J\nabla_X U_i,
\]
where $X$ is a tangent vector field to $N$, and we have used the fact that $\bar{M}^2(c)$ is a K\"ahler manifold. Hence, $\{U_1, U_2\}$ is a parallel orthonormal frame of tangent vector fields to $N$, which implies that $N$ is flat. On the other hand, for all $h\in H$ and all $i,j\in \{1,2\}$, we have
\[
h_*\II(U_i,U_j)=h_*(\bar{\nabla}_{U_i} U_j)^\perp=(h_*\bar{\nabla}_{U_i} U_j)^\perp=(\bar{\nabla}_{h_*U_i} h_*U_j)^\perp=\II(U_i,U_j).
\]
Then, the mean curvature vector field $\mathcal{H}=\sum_{i=1}^2\II(U_i,U_i)$ is $H$-equivariant and, hence, $\nabla^\perp$-parallel. In fact, for all $i,j,k\in \{1,2\}$ we have:
\[
(\nabla^\perp_{U_i}\II)(U_j,U_k)=\nabla^\perp_{U_i}\II(U_j,U_k)
-\II(\nabla_{U_i}U_j,U_k)-\II(U_j,\nabla_{U_i}U_k)=0,
\]
which implies that $N$ has parallel second fundamental form.

Now we show the converse. Let $N$ be a Lagrangian, flat, $2$-dimensional submanifold of $\bar{M}^2(c)$ with parallel mean curvature. If $N$ is minimal, it follows from the flatness of $N$ and the work~\cite{KM00} by Kenmotsu and Masuda, that $c>0$ and $N$ is an open part of a Clifford torus in $\C P^2$, that is, a minimal principal orbit of the canonical action of the group $H=U(1)\times U(1)\times U(1)$ on $\C P^2$, which is polar with cohomogeneity two.

From now on we assume that $N$ has nonzero parallel mean curvature. If $c>0$, it follows from the works of Ogata~\cite{Og95}, Kenmotsu and Zhou~\cite{KZ00}, and Hirakawa~\cite{Hi06}, that every totally real, flat surface of $\C P^2$ with nonzero parallel mean curvature is an open part of a flat torus in $\C P^2$, which, again, is a principal orbit of the polar action of $H=U(1)\times U(1)\times U(1)$. Let now $c<0$. In this case, it follows from the works of Kenmotsu and Zhou~\cite{KZ00}, and Hirakawa~\cite[Theorem~2]{Hi04} that $N$ is an open part of an orbit of the connected subgroup of $U(1,2)$ with Lie algebra $\R X\oplus\R \hat{Y}$, with
\[
X=
\small
\begin{pmatrix}
 0 & \sqrt{-c} & \sqrt{-c} \\
 \sqrt{-c} & 0 & 2 \sqrt{2} r+e^{-i s} r' \\
 \sqrt{-c} & -2 \sqrt{2} r-e^{i s} r' & 0
\end{pmatrix}\!
,
\hat{Y}=i
\small
\begin{pmatrix}
 0 & \sqrt{-c} & -\sqrt{-c} \\
 -\sqrt{-c} & 4 \sqrt{2} r & e^{-i s} r'-2 \sqrt{2} r \\
 \sqrt{-c} & e^{i s} r'-2 \sqrt{2} r & 4 \sqrt{2} r
\end{pmatrix}\!,
\]
where $r$ and $s$ are real numbers, and $r'=\sqrt{c+8r^2}\in \R$. Equivalently, $N$ is an open part of an orbit of the connected subgroup $H$ of $SU(1,2)$ with Lie algebra $\g{h}=\R X\oplus\R Y$, where we define $Y=\hat{Y}-ir\frac{8\sqrt{2}}{3}\id\in \g{g}=\g{su}(1,2)$, and where $\id$ denotes the identity matrix. Consider the standard Cartan involution $\theta$ on $\g{g}=\g{su}(1,2)$ defined by $\theta(X)=-X^*$, where $X^*$ denotes the conjugate transpose of a complex matrix. If we define $\g{k}=\{X\in\g{g}:\theta X=X\}$ and $\g{p}=\{X\in\g{g}:\theta X=-X\}$, we obtain a Cartan decomposition $\g{g}=\g{k}\oplus\g{p}$ of the simple Lie algebra $\g{g}$. Moreover, this decomposition is orthogonal with respect to the inner product on $\g{g}$ defined by $\langle X, Y\rangle=-B(\theta X, Y)$, where $X$, $Y\in \g{g}$, and $B$ is the Killing form of $\g{g}$. Note also that $\g{k}=\g{s}(\g{u}(1)\oplus\g{u}(2))\subset\g{su}(1,2)$ and $\g{p}$ consists of all matrices
\[
\begin{pmatrix}
    0 & \bar z_1 &  \bar z_2 \\
    z_1 & 0  & 0 \\
     z_2 & 0 & 0 \\
\end{pmatrix}\in \g{su}(1,2), \quad\textnormal{ with } z_1,z_2\in \C.
\]
Moreover, if $o\in \C H^2$ is the only fixed point of the group $K=S(U(1)U(2))$, then the differential of the map $g\in G=SU(1,2)\to g(o)\in\C H^2$ induces an isomorphism between $\g{p}$ and the tangent space $T_o\C H^2$. Define the two-dimensional totally real subspace $\g{s}$ of $\g{p}$ given by the relation $z_2=-\bar{z}_1$. Now we define the totally geodesic submanifold $\Sigma=\{(\Exp X) (o): X\in\g{s} \}$ of $\C H^2$, where $\Exp$ is the Lie group exponential map. In particular, the isomorphism between $T_o\C H^2$ and $\g{p}$ allows us to identify $T_o\Sigma$ with $\g{s}$. In order to conclude the proof, we will show that the action of $H$ on $\C H^2$ is polar with section $\Sigma$. Note first that, since $\g{h}$ is abelian and $G=SU(1,2)$ has rank $2$, $H$ is a closed subgroup of $G=SU(1,2)$.  According to \cite[Corollary~3.2]{BD11a}, it is sufficient to show that $T_o\Sigma\subset \nu_o(H\cdot o)$, $T_o\Sigma$ is a section for the slice representation of $H_o=\{h\in H:h(o)=o\}$ on $\nu_o(H\cdot o)$, and $\langle [\g{s},\g{s}],\g{h}\rangle=0$. It is straightforward to check that $\langle \g{s},\g{h}\rangle=0$, which implies that $T_o\Sigma$ is orthogonal to $T_o(H\cdot o)$. The second condition is also clearly satisfied, since the cohomogeneity of the slice representation coincides with the cohomogeneity of the $H$-action, which is $2$. Finally, again some elementary calculations show that $\langle [\g{s},\g{s}],\g{h}\rangle=0$. This proves that the action of $H$ on $\C H^2$ is polar, and hence $N$ is an open part of the principal orbit $H\cdot o$ of the action of $H$.
\end{proof}

\begin{remark}
Submanifolds with parallel second fundamental form are also an interesting topic of research nowadays. Naitoh tackled in~\cite{Na83a} and~\cite{Na83b} the classification problem of parallel submanifolds of complex space forms. An alternative proof of Theorem~\ref{th:principal} could be attempted using this classification.
\end{remark}

\subsection{Construction}\label{subsec:construction}

We proceed now with the construction of the new hypersurfaces with two principal curvatures in the nonflat complex space form $\bar{M}^2(c)$, $c\neq 0$. From now on, we fix a group $H$ of isometries of $\bar{M}^2(c)$ acting polarly and with cohomogeneity two on $\bar{M}^2(c)$. Let $\Sigma$ be a section, and $\Sigma_{reg}$ the set of regular points of $\Sigma$ for the $H$-action.

Let $\gamma\colon t\in (-\varepsilon, \varepsilon)\mapsto \gamma(t)\in \Sigma_{reg}$ be a curve in the regular part of $\Sigma$. Then, the subset
\[
H\cdot \gamma = \{h(\gamma(t)): t\in(-\varepsilon, \varepsilon), \, h\in H\}
\]
is a $3$-dimensional hypersurface in $\bar{M}^2(c)$ that is foliated orthogonally by equidistant $H$-orbits, and by the curves $h\circ \gamma\colon t\in (-\varepsilon, \varepsilon)\mapsto (h\circ \gamma)(t)=h(\gamma(t))\in \Sigma_{reg}$.
Our purpose is to determine which curves $\gamma$ give rise to hypersurfaces with exactly two principal curvatures at every point.

The first observation is that the principal curvatures of any principal $H$-orbit with respect to any nonzero normal vector are always different, that is, there are exactly two. This follows from the characterization in Theorem~\ref{th:principal} and from the explicit expression of the shape operator of Lagrangian, flat submanifolds with parallel mean curvature (see \cite[p.~299]{KZ00} or \cite[p.~232]{Hi06}; the expressions are also valid for the minimal case, cf.~\cite{KM00},~\cite{LOY75}):
\[
S_{e_3}=\begin{pmatrix}
r+\cos(s)\frac{r'}{2\sqrt{2}} & \sin(s)\frac{r'}{2\sqrt{2}}
\\
\sin(s)\frac{r'}{2\sqrt{2}} & 3r-\cos(s)\frac{r'}{2\sqrt{2}}
\end{pmatrix},\quad
S_{e_4}=\begin{pmatrix}
-\sin(s)\frac{r'}{2\sqrt{2}} & r+\cos(s)\frac{r'}{2\sqrt{2}}
\\
r+\cos(s)\frac{r'}{2\sqrt{2}} & \sin(s)\frac{r'}{2\sqrt{2}}
\end{pmatrix},
\]
where $\{e_3,e_4\}$ is an orthonormal basis of the normal space, $r$ and $s$ are real numbers (that depend only on the principal orbit), and $r'=\sqrt{c+8r^2}\in \R$. If we take a generic unit normal vector $\eta=\cos(\theta)e_3+\sin(\theta)e_4$, some elementary calculations show that the eigenvalues of the shape operator with respect to $\eta$ are
\[
2 r \cos (\theta )\pm\sqrt{\frac{\left(r'\right)^2}{8}+r^2-\frac{r r'}{\sqrt{2}} \cos (s+2 \theta )}.
\]
Since $c\neq 0$, it is not difficult to show that these principal curvatures are always different for every value of $r$, $s$, and $\theta$.

Fix now $p\in\Sigma_{reg}$, and $w\in T_p\Sigma_{reg}$. Consider a (locally defined) curve $\gamma$ in $\Sigma_{reg}$, parametrized by arc-length, and such that $\gamma(0)=p$, and $\dot{\gamma}(0)=w$. Fix a unit vector field $\xi$ along $\gamma$ tangent to $\Sigma_{reg}$, and such that $\langle \xi(t),\dot{\gamma}(t)\rangle=0$ for all $t$ where $\gamma$ is defined. Thus, $\xi(t)$ is a unit normal vector field to $H\cdot\gamma$ along $\gamma$. Consider also a local chart $\mathcal{U}$ for $\Sigma_{reg}$ around $p$, with coordinates $(x_1,x_2)$. Let $\alpha,\,\beta\colon T\mathcal{U}\to \R$ be the principal curvature functions of the principal orbits of $H$ intersecting $\mathcal{U}$ at the intersection points. As explained above, we know that $\alpha(\eta)\neq \beta(\eta)$ for any vector $\eta\in T\mathcal{U}$, and thus, $\alpha$ and $\beta$ are smooth functions.

The shape operator of $H\cdot \gamma$ at $\gamma(t)$ with respect to the unit normal vector $\xi(t)$ has the following eigenvalues:
\[
\alpha(\xi(t)),\ \beta(\xi(t)), \text{ and }
\langle S_{\xi(t)} \dot{\gamma}(t), \dot{\gamma}(t)\rangle=-\langle \bar{\nabla}_{\dot{\gamma}(t)}\xi, \dot{\gamma}(t)\rangle=
\langle \bar{\nabla}_{\dot{\gamma}(t)}\dot{\gamma}(t), \xi \rangle.
\]
This last eigenvalue is precisely the curvature of the curve $\gamma$ in $\bar{M}^2(c)$, or equivalently, since $\Sigma$ is totally geodesic, the curvature of $\gamma$ (with respect to the orientation determined by the normal field $\xi$) as a curve in $\Sigma$.

Hence, $H\cdot \gamma$ will have two principal curvatures at the points of $\gamma$ if and only if the curvature of $\gamma$ (with respect to $\xi$) as a curve in $\Sigma$ coincides with one of the two functions  $\alpha(\xi(t))$ or $\beta(\xi(t))$. If we write $\gamma$ in local coordinates as $\gamma(t)=(x_1(t),x_2(t))$, we have:
\begin{align*}
\bar{\nabla}_{\dot{\gamma}}\dot{\gamma}&=x_1''\partial_1
+x_1'\bar{\nabla}_{\dot{\gamma}}\partial_1+x_2''\partial_2
+x_1'\bar{\nabla}_{\dot{\gamma}}\partial_2\\
&=(x_1''+f(x_1,x_2,x_1',x_2'))\partial_1+(x_2''+g(x_1,x_2,x_1',x_2'))\partial_2,
\end{align*}
where $f$, $g$ are smooth functions depending on $x_1$, $x_2$, $x_1'$, $x_2'$ and the Christoffel symbols of~$\Sigma$. The fact that the curvature of $\gamma$ coincides with $\alpha(\xi(t))$ means that $\bar{\nabla}_{\dot{\gamma}}\dot{\gamma}=\alpha(\xi(t))\xi_{\gamma(t)}$. Hence, this condition is equivalent to the existence of smooth functions $F_\alpha$ and $G_\alpha$ such that
\begin{align*}
x_1''&=F_\alpha(x_1,x_2,x_1',x_2'),&
x_2''&=G_\alpha(x_1,x_2,x_1',x_2').
\end{align*}
This is a second order system of ordinary differential equations that has a unique solution for given initial conditions $\gamma(0)=p$, and $\dot{\gamma}(0)=w$. Therefore, the hypersurface $H\cdot \gamma$ has two principal curvatures at the points of $\gamma$ if and only if $\gamma$ is a solution to one of the two possible systems of ODEs constructed as explained above. A completely analogous argument applies for $\beta$ instead of $\alpha$, and it is obvious that the two possible choices that we have for $\gamma$, depending on whether we choose $\alpha$ or $\beta$ to be the principal curvature with multiplicity two, provide indeed two different curves, say $\gamma_\alpha\neq \gamma_\beta$, and thus, two different hypersurfaces $H\cdot\gamma_\alpha$ and $H\cdot\gamma_\beta$.

We still have to check that $H\cdot \gamma$ as constructed above has two principal curvatures at all points, not only along $\gamma$. An $H$-equivariant unit normal vector field to $H\cdot \gamma$ is given by $h_*\xi(t)$, for any $h\in H$ and any possible $t$. Note that this is a well-defined vector field because $H\cdot\gamma(t)$ is a principal orbit.
Since $H$ acts by isometries of $\bar{M}^2(c)$, the principal curvatures of $H\cdot \gamma$ at  $h(\gamma(t))$ with respect to $h_*\xi(t)$ are the same as the principal curvatures at $\gamma(t)$ with respect to $\xi(t)$.
We conclude that $H\cdot \gamma$ has exactly two principal curvatures at every point.

Finally, we have to show that the examples we have just constructed are indeed new, or equivalently, that their two principal curvatures are nonconstant. Fix a point $p\in\Sigma_{reg}$. Then we know that for every unit $w\in T_p\Sigma_{reg}$, there is a locally defined curve $\gamma_w$ such that $\gamma_w(0)=p$, $\dot{\gamma}_w(0)=w$, and $H\cdot \gamma_w$ has two principal curvatures. In fact, there are exactly two such curves, but the argument that follows applies to any of them. Assume that there is a collection of vectors $\g{w}$ in the unit sphere $S^1(T_p\Sigma_{reg})$ of $T_p\Sigma_{reg}$ generating $T_p\Sigma_{reg}$ and such that the real hypersurfaces $H\cdot \gamma_w$ are Hopf at $p$ for each $w\in\g{w}$. We will get a contradiction with this assumption. For each hypersurface $H\cdot\gamma_w$, $w\in\g{w}$, let $\xi_w$ be a unit normal vector field along $H\cdot \gamma_w$, which we know is $H$-equivariant along the principal $H$-orbits that foliate $H\cdot \gamma_w$. Note that the subindex $w$ in $\xi_w$ only denotes that the normal vector field depends on the initial value $w$ for $\dot{\gamma}_w$; in particular, $\langle \dot{\gamma}_w(t),(\xi_w )_{\gamma_w(t)}\rangle=0$ for each possible $t$. The assumption that $H\cdot \gamma_w$ is Hopf at~$p$ means that $(J\xi_w)_p$ is an eigenvector of the shape operator of $H\cdot\gamma_w$, and hence, $(J\xi_w)_p$ is also an eigenvector of the shape operator $S_{(\xi_w)_p}$ of the principal orbit $H\cdot p$ with respect to the normal vector $(\xi_w)_p$. In particular, the map
\[
w\in S^1(T_p\Sigma_{reg})\mapsto \langle S_{(\xi_w)_p} (J\xi_w)_p, Jw\rangle\in \R
\]
vanishes in $\g{w}$. Since this map is the restriction of a linear map to $S^1(T_p\Sigma_{reg})$ and $\g{w}$ generates $T_p\Sigma_{reg}$, it vanishes identically, which means that $(J\xi_w)_p$ is an eigenvector of $S_{(\xi_w)_p}$ for every unit $w\in T_p\Sigma_{reg}$. Since $Jw$ is perpendicular to $(J\xi_w)_p$ and $H\cdot\gamma_w$ is $2$-dimensional, we have that $Jw$ is an eigenvector of $S_{(\xi_w)_p}$ for each $w$ as well. But now, if we fix any $w$ and take unit normal vectors $\xi=(\xi_w)_p$ and $\eta=w$ at $p$, then $\{J\xi,Jw\}$ is a common basis of eigenvectors for the shape operators $S_\xi$ and $S_\eta$ of the principal orbit $H\cdot p$ at $p$ with respect to $\xi$ and~$\eta$. This means that the shape operators $S_\xi$ and $S_\eta$ commute. Using this and the fact that the principal orbit $H\cdot p$ has flat normal bundle~\cite[Corollary~3.2.5]{BCO03}, the Ricci equation of $H\cdot p$ applied to $J\xi$, $Jw$, $\xi$, and $\eta$ reads
\[
0=\langle R^\perp(J\xi,Jw)\xi,\eta\rangle=\langle\bar{R}(J\xi,Jw)\xi,\eta\rangle
+\langle[S_\xi,S_\eta]J\xi,J\eta\rangle=-\frac{c}{4},
\]
where $R^\perp$ is the normal curvature tensor of $H\cdot p$. This gives the desired contradiction. Therefore, the real hypersurfaces $H\cdot \gamma_w$ are Hopf at $p$ at most for two collinear vectors of $S^1(T_p\Sigma_{reg})$. So, generically, our hypersurfaces are non-Hopf. But since they have two principal curvatures, these cannot be constant, because all hypersurfaces in $\bar{M}^2(c)$ with two constant principal curvatures are Hopf, as follows from their well-known classification result (see~\cite{BD06} and~\cite{Ki86}).

Thus, we have proved the first part of our Main Theorem.

\begin{theorem}
Let $\bar{M}^2(c)$ be a complex space form of complex dimension $2$ and constant holomorphic curvature $c\neq 0$. Consider a polar action of a group $H$ acting with cohomogeneity two and with section $\Sigma$ on $\bar{M}^2(c)$.

Then, for any regular point $p\in \Sigma$ and any unit tangent vector $w\in T_p\Sigma$, there are exactly two different locally defined unit speed curves $\gamma_i\colon(-\epsilon,\epsilon)\to\Sigma$, $i=1,2$, with $\gamma_i(0)=p$ and $\dot{\gamma}_i(0)=w$, such that the set $H\cdot \gamma_i=\{h(\gamma_i(t)):h\in H,\, t\in(-\epsilon,\epsilon)\}$ is a real hypersurface with two principal curvatures in $\bar{M}^2(c)$. Generically, such hypersurface is non-Hopf and with nonconstant principal curvatures.
\end{theorem}


\section{Setup}\label{sect:setup}

From now on we are devoted to the classification part of the Main Theorem. We settle the notation in this section, and continue with the proof in the following sections.

Recall that we denote by $\bar{M}^2(c)$ a nonflat complex space form of complex dimension $2$ and constant holomorphic curvature $c\neq 0$. We denote by $\langle\,\cdot\,,\,\cdot\,\rangle$ the metric of $\bar{M}^2(c)$, by $J$ its complex structure, by $\bar{\nabla}$  its Levi-Civita connection, and by $\bar{R}$ its curvature tensor.

Let $M$ be a real hypersurface of $\bar{M}^2(c)$, that is, a submanifold of $\bar{M}^2(c)$ with real codimension $1$. The calculations that follow are local, so we may assume that $M$ is orientable. Thus, let $\xi$ be a unit normal vector field along $M$. The vector field $J\xi$ is tangent to $M$, and is called the \emph{Hopf} or \emph{Reeb vector field} of the hypersurface.

The Levi-Civita connection of $M$ is denoted by $\nabla$, and is determined by the Gauss formula
\[
\bar{\nabla}_X Y=\nabla_X Y+\langle S X,Y\rangle\xi,
\]
where $X$ and $Y$ are tangent vector fields along $M$, and $S$ denotes the shape operator of~$M$.

The shape operator is a self-adjoint endomorphism with respect to the metric of $M$, and thus it can be diagonalized with real eigenvalues. These eigenvalues are called the \emph{principal curvatures} of $M$, the corresponding eigenspaces are the principal curvature spaces, and the corresponding eigenvectors are the principal curvature vectors. Recall that $M$ is said to be a \emph{Hopf hypersurface} if $J\xi$ is a principal curvature vector field.

A real hypersurface of $\bar{M}^2(c)$ has, generically, three distinct principal curvatures. Tashiro and Tachibana~\cite{TT63} proved that there are no umbilical real hypersurfaces (that is, real hypersurfaces with exactly one principal curvature) in nonflat complex spaces forms. If $n\geq 3$, Cecil and Ryan~\cite{CR82} in $\C P^n$, and Montiel~\cite{Mo85} in~$\C H^n$ proved that real hypersurfaces with two distinct constant principal curvatures are open parts of homogeneous Hopf hypersurfaces.
The methods used by Cecil and Ryan, and Montiel, do not work if $n=2$, and we have indeed provided in Section~\ref{sect:construction} examples of real hypersurfaces in $\bar{M}^2(c)$ with two distinct nonconstant principal curvatures. The classification of real hypersurfaces with two distinct \emph{constant} principal curvatures in $\bar{M}^2(c)$ is not very difficult to obtain and can be found in~\cite{Ta75a} for $\C P^2$ as a particular case of the classification in arbitrary dimensions, and in~\cite{BD06}. Again, all the examples are open parts of homogeneous Hopf real hypersurfaces. The classification of Hopf real hypersurfaces with constant principal curvatures in nonflat complex space forms is due to Kimura~\cite{Ki86} in $\C P^n$, and to Berndt~\cite{Be89} in $\C H^n$.
\medskip

Therefore, our interest in this paper is to study non-Hopf real hypersurfaces with two nonconstant distinct principal curvatures in $\C P^2$ and $\C H^2$. Thus, let $\alpha$ and $\beta$ be the two distinct principal curvatures of a real hypersurface $M$. In a neighborhood of a point, $\alpha$ and $\beta$ will still be different functions, and therefore, the distributions $T_\alpha$ and $T_\beta$ consisting of the principal curvature spaces associated with $\alpha$ and $\beta$ respectively, are well-defined smooth distributions on this open neighborhood. By a similar continuity argument, in a maybe smaller neighborhood, the vector field $J\xi$ will not be a principal curvature vector at any point, and hence $M$ will be a non-Hopf real hypersurface in that neighborhood. In what follows we simply suppose that $M$ denotes that neighborhood where $\alpha$ and $\beta$ are different and $J\xi$ is not a principal curvature vector at any point. We will also assume without loss of generality that $\dim T_\alpha=1$, and $\dim T_\beta=2$. By $\Gamma(T_\alpha)$ and $\Gamma(T_\beta)$ we denote the sections of $T_\alpha$ and $T_\beta$, that is, the smooth vector fields on $M$ that are, at each point, principal curvature vectors associated with $\alpha$ and $\beta$ respectively.

The aim of this section is to prove the following algebraic result:

\begin{proposition}\label{prop:algebra}
There exist unit smooth vector fields $U\in\Gamma(T_\alpha)$, and $V$, $A\in\Gamma(T_\beta)$ with $V$ orthogonal to $A$, and positive smooth functions $a$, $b\colon M\to\R$ with $a^2+b^2=1$, such that
\begin{align*}
J\xi &= aU+bV, & JU  &= -bA-a\xi,\\
JA &= bU-aV, & JV  &= aA-b\xi.
\end{align*}
\end{proposition}

\begin{remark}\label{remark:phi}
Sometimes it will be more convenient to write $a=\cos\phi$, $b=\sin\phi$, for a smooth function $\phi\colon M\to (0,\pi/2)$.
\end{remark}

\begin{proof}
Since $J\xi$ is a unit vector field tangent to $M$, we can write $J\xi=a U+b V$, where $U\in\Gamma(T_\alpha)$, $V\in\Gamma(T_\beta)$ are unit vectors, and $a$, $b\colon M\to\R$ are smooth functions satisfying $a^2+b^2=1$, and $a,b>0$. Since $T_\beta$ is $2$-dimensional, we consider a unit vector field $A\in\Gamma(T_\beta)$ that is orthogonal to $V$. Then, $\{U,V,A\}$ forms an orthonormal basis of $TM$ at each point.

 Since $-\xi=J^2\xi=aJU+bJV$, and $a\neq 0$, taking inner product with $V$ it follows that $\langle JU,V\rangle=0$. This implies that $JU$, $JV\in\spann\{A,\xi\}$. Now, $\langle JU,\xi\rangle=-\langle U,J\xi \rangle=-a$, and since $U$ is a unit vector field, it follows that $\langle JU,A\rangle=\pm b$. By reversing the sign of $A$ if necessary, we may assume that $JU=-bA-a\xi$. A similar argument shows that $JV=aA-b\xi$. Finally, the previous expressions yield $\langle JA,U\rangle=b$, $\langle JA,V\rangle=-a$, and $\langle JA,\xi\rangle=0$, from where the result follows.
\end{proof}


\section{Levi-Civita connection}\label{sect:Levi-Civita}

Let $M$ be a real hypersurface of $\bar{M}^2(c)$, $c\neq 0$, with two distinct principal curvatures that is not Hopf at any point, and assume the notation given above by Proposition~\ref{prop:algebra} and Remark~\ref{remark:phi}. The aim of this section is to determine the Levi-Civita connection of $M$. This is established by the following proposition, whose proof will be given in the rest of this section.

\begin{proposition}\label{prop:Levi-Civita}
The Levi-Civita connection of $M$ in terms of the basis $\{U,V,A\}$ is given by the following equations:
\begin{align*}
\nabla_U U  &=-\frac{b(c-4\alpha(\alpha-\beta))}{4a(\alpha-\beta)}A,&
\nabla_U V  &=\frac{c}{4(\alpha-\beta)}A,\\
\nabla_V U  &=\frac{c}{4(\alpha-\beta)}A,&
\nabla_V V  &=-\frac{a(c-4\alpha(\alpha-\beta))}{4b(\alpha-\beta)}A,\\
\nabla_A U  &=-\frac{(a^2-2b^2)c}{4(\alpha-\beta)}V,&
\nabla_A V  &=\frac{(a^2-2b^2)c}{4(\alpha-\beta)}U,\\
\nabla_U A  &=\frac{b(c-4\alpha(\alpha-\beta))}{4a(\alpha-\beta)}U
    -\frac{c}{4(\alpha-\beta)}V,&
\nabla_V A  &=-\frac{c}{4(\alpha-\beta)}U
    +\frac{a(c+4\beta(\alpha-\beta))}{4b(\alpha-\beta)}V,\\[2ex]
\nabla_A A  &=0.\\
\end{align*}
Furthermore, we have $U\phi=V\phi=U\alpha=V\alpha=U\beta=V\beta=0$,
and
\begin{equation}\label{eq:differentialEquation}
\begin{aligned}
A\phi
&=\beta+\frac{c(1-3\sin^2\phi)}{4(\alpha-\beta)},\\
A\alpha &=\frac{1}{4}\bigl(c(2-3\sin^2\phi)+4\alpha(\alpha-\beta)\bigr)\tan\phi,\\
A\beta
&=-\frac{3c}{8}\sin 2\phi.
\end{aligned}
\end{equation}
\end{proposition}

Now we proceed with the proof of Proposition~\ref{prop:Levi-Civita}. It basically follows from the equations of a hypersurface in a complex space form, but the calculations that lead to it are long. In order to make things easier for the reader we will divide this proof in several steps.

We start by using the Codazzi equation of a hypersurface to get a first expression for the Levi-Civita connection $\nabla$ of $M$. Recall that for arbitrary vector fields $X$, $Y$, and $Z$ tangent to $M$, the Codazzi equation is written as
\[
\langle\bar{R}(X,Y)Z,\xi\rangle=
\langle(\nabla_X S)Y,Z\rangle-\langle(\nabla_Y S)X,Z\rangle.
\]

Let us apply the Codazzi equation to the triple $(U,A,U)$. Using the fact that $U$ and $A$ are orthogonal eigenvectors of $S$ associated with the eigenvalues $\alpha$ and $\beta$ respectively, and the symmetry of $S$ with respect to the inner product, we get
\begin{align*}
\langle(\nabla_U S)A,U\rangle
&{}=\langle\nabla_U SA-S\nabla_UA,U\rangle
=\langle\nabla_U(\beta A),U\rangle-\langle \nabla_UA,SU\rangle\\
&{}=(U\beta)\langle A,U\rangle+\beta\langle\nabla_UA,U\rangle
-\alpha\langle\nabla_UA,U\rangle
=(\alpha-\beta)\langle\nabla_UU,A\rangle.
\end{align*}
Since $U$ is a unit vector field we have $\langle\nabla_AU,U\rangle=0$. Taking this into account and proceeding as before we get $\langle(\nabla_A S)U,U\rangle=A\alpha$. Finally, the expression of the curvature tensor of a complex space form yields $\langle\bar{R}(U,A)U,\xi\rangle=-3abc/4$. Altogether this implies
\[
\langle\nabla_UU,A\rangle=\frac{1}{\alpha-\beta}
\left(A\alpha-\frac{3abc}{4}\right).
\]

Applying the Codazzi equation to the triples $(U,V,U)$, $(U,V,V)$, $(U,V,A)$, $(U,A,U)$, $(U,A,V)$, and $(U,A,A)$, we obtain in a similar way:
\begin{equation}\label{eq:nablas1}
\begin{aligned}
\langle\nabla_U U,V\rangle &{}=\frac{V\alpha}{\alpha-\beta},
&\langle\nabla_V U,V\rangle &{}=\frac{U\beta}{\alpha-\beta},
&\langle\nabla_V U,A\rangle &{}=\frac{c}{4(\alpha-\beta)},\\
\langle\nabla_U U,A\rangle &{}=\frac{1}{\alpha-\beta}\Bigl(A\alpha-\frac{3abc}{4}\Bigr),
&\langle\nabla_A U,V\rangle &{}=\frac{c(2b^2-a^2)}{4(\alpha-\beta)},
&\langle\nabla_A U,A\rangle &{}=\frac{U\beta}{\alpha-\beta}.
\end{aligned}
\end{equation}

Since $J$ is parallel with respect to the connection $\bar\nabla$, we have $\bar{\nabla}_U J\xi=J\bar{\nabla}_U\xi=-JSU=-\alpha JU$. Taking this into account, and using Proposition~\ref{prop:algebra} and~\eqref{eq:nablas1} we get
\begin{align*}
0\!&=U\langle A,J\xi\rangle=\langle\bar{\nabla}_U A,J\xi\rangle+\langle A,\bar{\nabla}_U J\xi\rangle
=a\langle\nabla_U A,U\rangle+b\langle\nabla_U A,V\rangle+\alpha b\langle A,A\rangle+\alpha a\langle A,\xi\rangle\\
&=-\frac{a}{\alpha-\beta}\Bigl(A\alpha-\frac{3abc}{4}\Bigr)+b\langle\nabla_U A,V\rangle+\alpha b,
\end{align*}
from where we obtain $\langle\nabla_U A,V\rangle$. This, and an analogous argument with $V\langle A,J\xi\rangle=0$ and $A\langle A,J\xi\rangle=0$, give the three expressions
\begin{equation}\label{eq:nablas2}
\begin{aligned}
\langle\nabla_U V,A\rangle &{}=\alpha-\frac{a}{b(\alpha-\beta)}\Bigl(A\alpha-\frac{3abc}{4}\Bigr),
&\langle\nabla_V V,A\rangle &{}=-\frac{a}{b}\Bigl(\beta+\frac{c}{4(\alpha-\beta)}\Bigr),\\
\langle\nabla_A V,A\rangle &{}=-\frac{aU\beta}{b(\alpha-\beta)}.
\end{aligned}
\end{equation}

Therefore, \eqref{eq:nablas1} and~\eqref{eq:nablas2} imply
\begin{equation}\label{eq:Levi-Civita1}
\begin{gathered}
\begin{aligned}
\nabla_U U  &=\frac{V\alpha}{\alpha-\beta}V+\frac{1}{\alpha-\beta}\Bigl(A\alpha-\frac{3abc}{4}\Bigr)A,
&\nabla_V U  &=\frac{U\beta}{\alpha-\beta}V+\frac{c}{4(\alpha-\beta)}A,\\
\nabla_A U  &=\frac{c(2b^2-a^2)}{4(\alpha-\beta)}V+\frac{U\beta}{\alpha-\beta}A,
&\nabla_A V  &=-\frac{c(2b^2-a^2)}{4(\alpha-\beta)}U-\frac{aU\beta}{b(\alpha-\beta)}A,\\
\nabla_V A  &=-\frac{c}{4(\alpha-\beta)}U+\frac{a}{b}\Bigl(\beta+\frac{c}{4(\alpha-\beta)}\Bigr)V,
&\nabla_A A  &=-\frac{U\beta}{\alpha-\beta}U+\frac{aU\beta}{b(\alpha-\beta)}V,
\end{aligned}\\
\begin{aligned}
\nabla_U V  &=-\frac{V\alpha}{\alpha-\beta}U+
\Bigl(\alpha-\frac{a}{b(\alpha-\beta)}\Bigl(A\alpha-\frac{3abc}{4}\Bigr)\Bigr)A,\\
\nabla_V V  &=-\frac{U\beta}{\alpha-\beta}U-\frac{a}{b}\Bigl(\beta+\frac{c}{4(\alpha-\beta)}\Bigr)A,\\
\nabla_U A  &=-\frac{1}{\alpha-\beta}\Bigl(A\alpha-\frac{3abc}{4}\Bigr)U
-\Bigl(\alpha-\frac{a}{b(\alpha-\beta)}\Bigl(A\alpha-\frac{3abc}{4}\Bigr)\Bigr)V.
\end{aligned}
\end{gathered}
\end{equation}

The next step is to calculate the derivatives of the functions $a$ and $b$ in terms of the derivatives of $\alpha$ and $\beta$. For example, using Proposition~\ref{prop:algebra} and~\eqref{eq:nablas2} we get
\[
Aa=A\langle U,J\xi\rangle=\langle\bar{\nabla}_A U,J\xi\rangle+\langle U,\bar{\nabla}_A J\xi\rangle
=b\langle\nabla_A U,V\rangle-\beta\langle U,JA\rangle
=\frac{bc(2b^2-a^2)}{4(\alpha-\beta)}-b\beta.
\]
We directly give the results for the other derivatives, whose calculations are similar to those of $Aa$:
\begin{equation}\label{eq:ab}
\begin{aligned}
Ua  &=\frac{bV\alpha}{\alpha-\beta},
&Va &=\frac{bU\beta}{\alpha-\beta},
&Aa &=\frac{bc(2b^2-a^2)}{4(\alpha-\beta)}-b\beta,\\
Ub  &=-\frac{aV\alpha}{\alpha-\beta},
&Vb &=-\frac{aU\beta}{\alpha-\beta},
&Ab &=-\frac{ac(2b^2-a^2)}{4(\alpha-\beta)}+a\beta.
\end{aligned}
\end{equation}

Now we calculate the derivatives of $\alpha$ and $\beta$. We start with $\beta$. First, the Codazzi equation implies
\begin{align*}
0   &=  \langle\bar{R}(V,A)A,\xi\rangle
=\langle(\nabla_V S)A,A\rangle-\langle(\nabla_A S)V,A\rangle\\
&=\langle \nabla_V (\beta A)-S\nabla_V A,A\rangle
-\langle \nabla_A(\beta V)-S\nabla_A V,A\rangle\\
&=V\beta+\beta\langle\nabla_V A,A\rangle-\beta\langle\nabla_V A,A\rangle
-\beta\langle\nabla_A V,A\rangle+\beta\langle\nabla_A V,A\rangle=V\beta.
\end{align*}
Similarly, using the expression for the curvature tensor together with Proposition~\ref{prop:algebra} we obtain
\begin{align*}
\frac{3abc}{4}
&=  \langle\bar{R}(V,A)V,\xi\rangle
=\langle(\nabla_V S)A,V\rangle-\langle(\nabla_A S)V,V\rangle\\
&=\langle \nabla_V (\beta A)-S\nabla_V A,V\rangle
-\langle \nabla_A(\beta V)-S\nabla_A V,V\rangle=-A\beta.
\end{align*}
Thus, we have $V\beta=0$, and $A\beta=-{3abc}/{4}$.

The next step is to calculate $U\beta$, but this will take some effort. Using~\eqref{eq:Levi-Civita1}, $V\beta=0$, and $A\beta=-{3abc}/{4}$, we get
\begin{align*}
[A,V]\beta &=(\nabla_A V-\nabla_V A)\beta
=-\frac{c(2b^2-a^2)}{4(\alpha-\beta)}U\beta+\frac{3abc}{4}\frac{a}{b(\alpha-\beta)}U\beta
+\frac{c}{4(\alpha-\beta)}U\beta\\
&=\frac{c(1+4a^2-2b^2)}{4(\alpha-\beta)}U\beta.
\end{align*}

On the other hand, since $V\beta=0$, it is clear that $AV\beta=0$, and using $A\beta=-3abc/4$ and~\eqref{eq:ab}, we also obtain
\[
VA\beta=-\frac{3c}{4}V(ab)=-\frac{3c}{4}((Va)b+a(Vb))
=\frac{3c(a^2-b^2)}{4(\alpha-\beta)}U\beta.
\]
Taking all this together we have
\begin{equation}\label{eq:Ubeta}
0=([A,V]-AV+VA)\beta=\frac{c(1+7a^2-5b^2)}{4(\alpha-\beta)}U\beta.
\end{equation}

Assume momentarily that there exists a point $p\in M$ such that $1+7a(p)^2-5b(p)^2=0$. Taking the derivative with respect to $V$ in the previous equation and evaluating at $p$ yields
\begin{align*}
0&=V_p\Bigl(\frac{c(1+7a^2-5b^2)}{4(\alpha-\beta)}U\beta\Bigr)\\
&=\frac{c}{4}\Bigl(\frac{V(1+7a^2-5b^2)}{\alpha-\beta}U\beta
-(1+7a^2-5b^2)\frac{V(\alpha-\beta)}{(\alpha-\beta)^2}U\beta
+\frac{1+7a^2-5b^2}{\alpha-\beta}VU\beta\Bigr)\Bigr\vert_p\\
&=\frac{c}{4(\alpha(p)-\beta(p))}V_p(1+7a^2-5b^2)(U_p\beta).
\end{align*}
Now, using~\eqref{eq:ab} we get
\[
0=\frac{c}{4(\alpha(p)-\beta(p))}\Bigl(14a\frac{b}{\alpha-\beta}
+10b\frac{a}{\alpha-\beta}\Bigr)(p)(U_p\beta)^2
=\frac{6a(p)b(p)c}{(\alpha(p)-\beta(p))^2}(U_p\beta)^2,
\]
and since by assumption $a,b\neq0$, we finally get $U_p\beta=0$. Therefore, \eqref{eq:Ubeta} implies $U\beta=0$ and thus we have
\begin{align}\label{eq:beta}
U\beta=V\beta  &=0,    &   A\beta  &=-\frac{3abc}{4}.
\end{align}

Now, by~\eqref{eq:Levi-Civita1} and $U\beta=V\beta=0$, it follows that $[A,U]\beta=0$, and $AU\beta=0$. Hence, by~\eqref{eq:beta} and~\eqref{eq:ab},
\[
0=([A,U]-AU+UA)\beta=U\Bigl(-\frac{3abc}{4}\Bigr)=-\frac{3c(b^2-a^2)}{4(\alpha-\beta)}V\alpha.
\]
Arguing as before, if $p\in M$ is such that $b(p)^2-a(p)^2=0$, then using again~\eqref{eq:ab} we get
\[
0=U_p\Bigl(-\frac{3c(b^2-a^2)}{4(\alpha-\beta)}V\alpha\Bigr)
=\frac{3a(p)b(p)c}{(\alpha(p)-\beta(p))^2}(V_p\alpha)^2,
\]
and the two previous equations readily imply $V\alpha=0$ on $M$.

Equations~\eqref{eq:Levi-Civita1} and~\eqref{eq:beta} yield
\[
0=([U,V]-UV+VU)\beta=-\frac{3abc}{4}\Bigl(
\Bigl(\alpha-\frac{a}{b(\alpha-\beta)}\Bigl(A\alpha-\frac{3abc}{4}\Bigr)\Bigr)
-\frac{c}{4(\alpha-\beta)}\Bigr),
\]
from where we can obtain $A\alpha$. Indeed, we have
\begin{align}\label{eq:alpha}
V\alpha &=0, &   A\alpha &=\frac{b}{4a}\Bigl(c(3a^2-1)+4\alpha(\alpha-\beta)\Bigr).
\end{align}

In order to obtain $U\alpha$ we need the Gauss equation:
\[
\langle\bar{R}(X,Y)Z,W\rangle=\langle{R}(X,Y)Z,W\rangle
+\langle SX,Z\rangle\langle SY,W\rangle-\langle SX,W\rangle\langle SY,Z\rangle,
\]
where $X$, $Y$, $Z$ and $W$ are vector fields tangent to $M$, and $R$ is the curvature tensor of~$M$.

Thus, let us use the Gauss equation for the tuple $(U,V,U,A)$. Proposition~\ref{prop:algebra} implies $\langle\bar{R}(U,V)U,A\rangle=0$, and $\langle SU,U\rangle\langle SV,A\rangle-\langle SU,A\rangle\langle SV,U\rangle=0$.

Using~\eqref{eq:Levi-Civita1} together with~\eqref{eq:beta} and~\eqref{eq:alpha} we get
\[
\langle\nabla_U\nabla_V U,A\rangle
=\langle\nabla_U\Bigl(\frac{c}{4(\alpha-\beta)}A\Bigr),A\rangle
=U\Bigl(\frac{c}{4(\alpha-\beta)}\Bigr)
=-\frac{c U(\alpha-\beta)}{4(\alpha-\beta)^2}
=-\frac{c}{4(\alpha-\beta)^2}U\alpha.
\]

Now, substituting~\eqref{eq:alpha} in~\eqref{eq:Levi-Civita1} yields
\[
\nabla_U U=-\frac{b(c-4\alpha(\alpha-\beta))}{4a(\alpha-\beta)}A.
\]
Taking into account that $V\alpha=V\beta=Va=Vb=0$ by equations~\eqref{eq:alpha},~\eqref{eq:beta} and~\eqref{eq:ab}, we get
\[
\langle\nabla_V\nabla_U U,A\rangle
=\langle\nabla_V\Bigl(-\frac{b(c-4\alpha(\alpha-\beta))}{4a(\alpha-\beta)}A\Bigr),A\rangle
=V\Bigl(-\frac{b(c-4\alpha(\alpha-\beta))}{4a(\alpha-\beta)}\Bigr)=0.
\]

Finally, since $U\beta=V\alpha=0$, equation~\eqref{eq:Levi-Civita1} implies that $[U,V]$ is a multiple of $A$. Since $\langle\nabla_A U,A\rangle=0$, it readily follows that $\langle\nabla_{[U,V]}U,A\rangle=0$.

Altogether this means that the Gauss equation is equivalent to
\[
-\frac{c}{4(\alpha-\beta)^2}U\alpha=0,
\]
from where it follows that $U\alpha=0$.

Putting together all the results of this section, and taking into account that $a=\cos\phi$, $b=\sin\phi$,  we obtain Proposition~\ref{prop:Levi-Civita}, as desired.


\section{Two perpendicular integrable distributions}\label{sect:distributions}

In this section we will study the properties of two integrable distributions: the one generated by $A$, and the distribution generated by the vector fields $U$ and $V$. This information will be used in Section~\ref{sect:proof} to finish the proof of the Main Theorem. We assume the notation and results obtained so far, and summarized in propositions~\ref{prop:algebra} and~\ref{prop:Levi-Civita}.

We start with the distribution $\mathcal{D}$ generated by the vector fields $U$ and $V$.

\begin{proposition}\label{prop:distribution}
The distribution $p\mapsto\mathcal{D}_p=\spann\{U_p,V_p\}$ of $M$ is integrable, and its leaves are flat, totally real submanifolds of $\bar{M}^2(c)$ with parallel second fundamental form and flat normal bundle.
\end{proposition}

\begin{proof}
From Proposition~\ref{prop:Levi-Civita} it follows, using the Gauss formula, that
\begin{equation}\label{eq:distribution}
\begin{aligned}
\bar{\nabla}_U U  &=-\frac{b(c-4\alpha(\alpha-\beta))}{4a(\alpha-\beta)}A+\alpha\xi,&
\bar{\nabla}_V V  &=-\frac{a(c-4\alpha(\alpha-\beta))}{4b(\alpha-\beta)}A+\beta\xi,\\
\bar{\nabla}_U V  &=\bar{\nabla}_V U =\frac{c}{4(\alpha-\beta)}A.
\end{aligned}
\end{equation}
From here it immediately follows that $[U,V]=0$, and thus, the distribution $\mathcal{D}$ is integrable.

Let $L$ be any integral submanifold of this distribution, and denote by $\widetilde{\nabla}$ its Levi-Civita connection. The normal space $\nu_p L$ of $L$ at $p\in L$, as a submanifold of $\bar{M}^2(c)$, is generated by $A_p$ and $\xi_p$. Hence, using~\eqref{eq:distribution} we get
$\widetilde{\nabla}_U U=\widetilde{\nabla}_U V=\widetilde{\nabla}_V U=\widetilde{\nabla}_V V=0$. Therefore, $U$ and $V$ are parallel vector fields on $L$ with respect to the Levi-Civita connection $\widetilde{\nabla}$. In particular, $L$ is flat as a $2$-dimensional Riemannian manifold.

Let $\nabla^\perp$ denote the normal connection of $L$ as a submanifold of $\bar{M}^2(c)$. Since $\bar\nabla_U\xi=-SU=-\alpha U$, and $\bar\nabla_V\xi=-SV=-\beta V$, if follows that $\nabla^\perp_U\xi=\nabla^\perp_V\xi=0$. On the other hand, the formulas for $\nabla_U A$ and $\nabla_V A$ in Proposition~\ref{prop:Levi-Civita} yield $\nabla^\perp_U A=\nabla^\perp_V A=0$. Therefore, the normal bundle of $L$ is flat, and indeed, $A$ and $\xi$ are parallel normal vector fields along $L$.

Finally, Proposition~\ref{prop:Levi-Civita} implies that the functions $\phi$, $\alpha$, $\beta$, and hence also $a$ and $b$, are constant along the integral curves of $U$ and $V$. Therefore, they are constant along $L$, and hence their derivatives vanish. This, the fact that $U$ and $V$ are parallel with respect to $\widetilde\nabla$, and the fact that $A$ and $\xi$ are parallel with respect to $\nabla^\perp$, imply that the second fundamental form of $L$ is parallel.
\end{proof}

Thus, the real hypersurface $M$ is foliated orthogonally by the leaves of the $2$-dimensional distribution $\mathcal{D}=\spann\{U,V\}$, and by the integral curves of the vector field $A$. We now provide more information on the integral curves of $A$.

\begin{proposition}\label{prop:gamma}
Let $\gamma$ be an integral curve of $A$ through a point $p\in M$. Let $Q_p=\exp_p(\R A_p\oplus\R\xi_p)$, where $\exp_p$ denotes the Riemannian exponential map of $\bar{M}^2(c)$ at $p$. Then, $Q_p$ is a totally real, totally geodesic submanifold of $\bar{M}^2(c)$, and $\gamma$ is contained in $Q_p$.

Moreover, the curve $\gamma$ is determined by the initial conditions $\gamma(0)=p$, $\dot\gamma(0)=A_p$, and the fact that $\gamma$ is a unit speed curve in $Q_p=\exp_p(\R A_p\oplus\R\xi_p)$ with curvature $\tilde\beta$ with respect to $\xi$, where $\tilde\beta$ is a function determined by the differential equation
\begin{align*}
\tilde\phi'
&=\tilde\beta+\frac{c(1-3\sin^2\tilde\phi)}{4(\tilde\alpha-\tilde\beta)},\\
\tilde\alpha' &=\frac{1}{2}\bigl(c(2-3\sin^2\tilde\phi)
+4\tilde\alpha(\tilde\alpha-\tilde\beta)\bigr)\tan\tilde\phi,\\
\tilde\beta'
&=-\frac{3c}{8}\sin 2\tilde\phi.
\end{align*}
with initial conditions $\tilde\phi(0)=\phi(p)$, $\tilde\alpha(0)=\alpha(p)$, and $\tilde\beta(0)=\beta(p)$.
\end{proposition}

\begin{proof}
From Proposition~\ref{prop:algebra} it is clear that $\langle JA,\xi\rangle=0$, and hence $\R A_p\oplus\R\xi_p$ is a totally real subspace of $T_p\bar{M}^2(c)$, that is, $J(\R A_p\oplus\R\xi_p)$ is orthogonal to $\R A_p\oplus\R\xi_p$. Then, it is well-known that $Q_p$ is a totally real, totally geodesic submanifold of $\bar{M}^2(c)$. We now prove that the curve $\gamma$ is contained in $Q_p$.

From the Gauss equation and Proposition~\ref{prop:Levi-Civita} it follows that
\[
\bar{\nabla}_AA=\beta\xi,\quad
\bar{\nabla}_A\xi=-\beta A.
\]

As a consequence, the curvature of $\gamma$ with respect to the vector field $\xi$ is given by  $\kappa[\gamma](t)=\langle\bar\nabla_{\dot\gamma(t)}\dot\gamma,\xi_{\gamma(t)}\rangle
=(\beta\circ\gamma)(t)$. It follows from Proposition~\ref{prop:Levi-Civita} that the functions $\phi$, $\alpha$, and $\beta$ are constant along the leaves of the distribution $\mathcal{D}$. Thus, equation~\eqref{eq:differentialEquation} translates into the differential equation given in the statement of Proposition~\ref{prop:gamma}. In particular, the curvature of $\gamma$ is given by the function $\tilde\beta(t)=(\beta\circ\gamma)(t)$. Therefore, the curve $\gamma$ is determined by the differential equation
\begin{equation}\label{eq:gamma}
\bar{\nabla}_{\dot\gamma}\dot\gamma=\tilde\beta\tilde\xi,\quad
\bar{\nabla}_{\dot\gamma}\tilde\xi=-\tilde\beta\dot\gamma,\quad
\gamma(0)=p, \quad\dot\gamma(0)=A_p, \quad\tilde\xi(0)=\xi_p.
\end{equation}
Now, since $Q_p$ is totally geodesic, the Levi-Civita connection of $Q_p$ coincides with the restriction of the Levi-Civita connection of $\bar{M}^2(c)$ to $Q_p$. Hence, the restriction of~\eqref{eq:gamma} to $Q_p$ is well-defined and has a solution in $Q_p$. This solution is thus a solution to~\eqref{eq:gamma} as well, and by uniqueness of solutions to ordinary differential equations both curves must be the same. Therefore, we have proved that $\gamma$ is contained in $Q_p$. (Another way of proving this fact is to use~\cite[Theorem~3.4]{Do75}.)

Finally, since $\gamma$ is contained in $Q_p$, and $Q_p$ is a complete $2$-dimensional Riemannian manifold of constant sectional curvature (because it is totally real and totally geodesic in $\bar{M}^2(c)$), then the curve $\gamma$ is determined by its curvature, an initial point, an initial tangent vector, and a choice of orientation given in this case by $\xi$. Hence, Proposition~\ref{prop:gamma} follows.
\end{proof}

We study some further properties of the integral submanifolds of the integrable distribution $\mathcal{D}$.

\begin{lemma}\label{lemma:intersection}
Let $p\in M$, $Q_p=\exp_p(\R A_p\oplus\R\xi_p)$, and $\gamma$ an integral curve of $A$ through $p$. Then, $Q_p$ intersects the integral submanifolds of $\mathcal{D}$ through $\gamma(t)$ perpendicularly.
\end{lemma}

\begin{proof}
By Proposition~\ref{prop:gamma}, $\gamma$ is contained in $Q_p$. Clearly, $A_{\gamma(t)}=\dot\gamma(t)$ is tangent to $Q_p$ along $\gamma$. We now show that $\xi_{\gamma(t)}$ is tangent to $Q_p$. Let $\eta$ be a vector field along $\gamma$ such that $\eta_{p}\in\nu_{p}Q_p$ and that is parallel with respect to the normal connection $D^\perp$ of $Q_p$. Then, since $Q_p$ is totally geodesic, the Weingarten formula implies $\bar\nabla_{\dot\gamma}\eta=D^\perp_{\dot\gamma}\eta=0$. Hence,
\[
\frac{d}{dt}\langle\xi,\eta\rangle=\langle\bar\nabla_{\dot\gamma}\xi,\eta\rangle
+\langle\xi,\bar\nabla_{\dot\gamma}\eta\rangle
=-\beta\langle\dot\gamma,\eta\rangle=0,
\]
and since $\langle\xi_{\gamma(0)},\eta_{\gamma(0)}\rangle=0$, and $\eta$ is arbitrary, we conclude that $\xi_{\gamma(t)}$ is tangent to $Q_p$ for all $t$.

Altogether this implies that $T_{\gamma(t)}Q_p=\spann\{A_{\gamma(t)},\xi_{\gamma(t)}\}$. Therefore, by construction we have $Q_p=Q_{\gamma(t)}$ for all $t$, and in particular, $Q_p$ is perpendicular to the leaf of $\mathcal{D}$ through $\gamma(t)$, as we wanted to show.
\end{proof}

We are now in position to state the key result of this section.

\begin{proposition}\label{prop:equidistant}
We have:
\begin{enumerate}[\rm (i)]
\item The integral submanifolds of $\mathcal{D}$ are equidistant submanifolds of $\bar{M}^2(c)$.

\item Let $L$ be an integral submanifold of the distribution $\mathcal{D}$, and let $L_t$ be an integral submanifold of $\mathcal{D}$ whose distance to $L$ is a sufficiently small number $t$. Then, in a neighborhood $\mathcal{U}$ of a point of $L$ there exists a parallel normal vector field $\eta_t$ such that
    \[
    L_t=\{\exp_p(\eta_t(p)):p\in \mathcal{U}\}.
    \]
\end{enumerate}
\end{proposition}

\begin{proof}
Let $L$ be a leaf of the distribution $\mathcal{D}$. We continue to denote by $\nabla^\perp$ the normal connection of $L$. Recall from Proposition~\ref{prop:distribution} that $\nabla^\perp$ is flat, and indeed, $\{A,\xi\}$ constitutes a parallel basis of the normal bundle $\nu L$ of $L$ as a submanifold of $\bar{M}^2(c)$.

Let $p\in L$, and let $\gamma_p$ be an integral curve of $A$ through $p$. We denote by $L_t$ the integral manifold of $\mathcal{D}$ through $\gamma_p(t)$. Since $A$ is a geodesic vector field in $M$ by Proposition~\ref{prop:Levi-Civita}, $\gamma_p$~is a unit speed geodesic in $M$. Since $\gamma_p$ is perpendicular to $L$ and $L_t$, $d_M(L,L_t)=\mathcal{L}({\gamma_p}_{\mid[0,t]})$, where $d_M$ denotes the Riemannian distance of $M$, and $\mathcal{L}(\cdot)$ is the length of a curve. The point $p$ is arbitrary, and thus the integral submanifolds of $\mathcal{D}$ are equidistant in $M$.

For a point $p\in L$, we consider the geodesic $\sigma_p$ of $\bar{M}^2(c)$ that minimizes the distance between $\sigma_p(0)=p$ and $\sigma_p(1)=\gamma_p(t)$. Since $Q_p=\exp_p(\R A_p\oplus\R\xi_p)$ is totally geodesic, $\gamma_p$~is contained in~$Q_p$ (see Proposition~\ref{prop:gamma}), and $t$ is sufficiently small, it follows that $\sigma_p$ is contained in~$Q_p$. Since $Q_p$ intersects $L$ and $L_t$ orthogonally by Lemma~\ref{lemma:intersection}, so does $\sigma_p$ and we conclude that $\sigma_p$ is a minimizing geodesic of $\bar{M}^2(c)$ between these two submanifolds. We define $\eta_t(p)=\dot\sigma_p(0)\in\R A_p\oplus\R \xi_p$ for an arbitrary $p\in L$.

Now let $q\in L$ be another point. By the previous argument the curve $\gamma_q$, if defined for time $t$, realizes the distance in $M$ between $L$ and $L_t$. The normal bundle of $L$ in $\bar{M}^2(c)$ is flat with respect to the normal connection $\nabla^\perp$, and $\{A,\xi\}$ is parallel. Thus, $Q_q$ is obtained by the parallel transport of $Q_p$ to $q$ along $L$.
There is a unique isometry $g$ of $\bar{M}^2(c)$ such that $g(p)=q$, $g_{*p}(A_p)=A_q$, and $g_{*p}(\xi_p)=\xi_q$, where $g_*$ denotes the differential of $g$ (note that $\R A\oplus\R \xi$ is totally real). Since $g$ is an isometry, and the curvature of $\gamma_q$ is given by the same function as that of $\gamma_p$ due to Proposition~\ref{prop:gamma}, it follows that the curve $g\circ\gamma_p$ satisfies the differential equation~\eqref{eq:gamma} with $q$ instead of $p$. By uniqueness, we have that $g\circ\gamma_p=\gamma_q$. Hence, the geodesic $\sigma_q$ that minimizes the distance between $q$ and $\gamma_q(t)$ coincides with $g\circ\sigma_p$, and thus it satisfies that $\dot\sigma_q(0)=g_{*p}(\dot\sigma_p(0))$. Since parallel transport is a linear isometry, and there is a unique isometry between $\nu_p L$ and $\nu_q L$ mapping $A_p$ to $A_q$ and $\xi_p$ to $\xi_q$, it follows that $\dot\sigma_q(0)$ is precisely the $\nabla^\perp$-parallel transport of $\dot\sigma_p(0)$ to $q$. Therefore, $\eta_t$ is a normal parallel vector field along $L$ wherever it is defined, and we conclude that $L_t=\{\exp_p(\eta_t(p)):p\in L\}$ (at least locally). This proves Proposition~\ref{prop:equidistant}.
\end{proof}


\section{Proof of the Main Theorem}\label{sect:proof}

Let $M$ be a real hypersurface of $\bar{M}^2(c)$ with two distinct principal curvatures. We further assume that $M$ is non-Hopf at every point, which means that $a$, $b\neq 0$ along $M$, in the notation of Section~\ref{sect:setup}. Assume also the notation of propositions~\ref{prop:algebra} and~\ref{prop:Levi-Civita} as usual. We have seen in Section~\ref{sect:distributions} that $M$ is foliated orthogonally by the integral submanifolds of the integrable $2$-dimensional distribution $\mathcal{D}=\spann\{U,V\}$, and by the integral curves of the vector field $A$.

The integral submanifolds of $\mathcal{D}$ are $2$-dimensional totally real, flat submanifolds with parallel second fundamental form and flat normal bundle in $\bar{M}^2(c)$. In particular, they are surfaces with parallel mean curvature in $\bar{M}^2(c)$. Theorem~\ref{th:principal} then guarantees that each integral submanifold of $\mathcal{D}$ is an open part of a principal orbit of a polar action of cohomogeneity two on $\bar{M}^2(c)$. In order to conclude the proof of the Main Theorem we have to show that all integral surfaces of $\mathcal{D}$ are open parts of principal orbits of the \emph{same} polar action.

Let $H\cdot p$ be a principal orbit of the action of $H$, and let $\eta$ be an $H$-equivariant normal vector field along $H\cdot p$. If $\nabla^\perp$ denotes now the normal connection of $H\cdot p$, then $\nabla^\perp\eta=0$, as equivariant vector fields are parallel~\cite[Corollary~3.2.5]{BCO03}. It is then clear that equivariant vector fields along $H\cdot p$ are in one-to-one correspondence with $\nabla^\perp$-parallel normal vector fields along $H\cdot p$. If $q\in\bar{M}^2(c)$, then there exists a minimizing geodesic $\sigma$ from $p$ to $H\cdot q$ which intersects both $H\cdot p$ and $H\cdot q$ orthogonally. We may assume that $\sigma(0)=p$, $\sigma(1)=q$, define $\eta_p=\dot\sigma(0)$, and extend $\eta_p$ to a normal parallel vector field $\eta$ along $H\cdot p$. Using that $\eta$ is also equivariant, it is then easy to obtain
\[
H\cdot q=\{h(\exp_p(\eta_p)):h\in H\}=\{\exp_{h(p)}(\eta_{h(p)}):h\in H\}
=\{\exp_x(\eta_x):x\in H\cdot p\}.
\]

We saw above that an integral submanifold $L$ of the distribution $\mathcal{D}$ is, up to holomorphic congruence, an open part of a principal orbit of the action of $H$ on $\bar{M}^2(c)$. According to Proposition~\ref{prop:equidistant}, the rest of the sufficiently close integral submanifolds of $\mathcal{D}$ are obtained locally as $\{\exp_q(\eta_t(q)):q\in \mathcal{U}\}$, where $\mathcal{U}$ is an open subset of $L$, and $\eta_t$ is a suitable parallel normal vector field along $\mathcal{U}\subset L$. Hence, all the integral submanifolds of $\mathcal{D}$ are open parts of principal orbits of the action of $H$. If $\gamma$ is an integral curve of the vector field~$A$, then it is obvious that in a neighborhood of the point $p$, the hypersurface $M$ is obtained as $M=H\cdot\gamma$, as stated in the Main Theorem. This concludes the proof.

We still have to justify the assertions made in Remark~\ref{remark:ODE}, which basically follow from the construction in Subsection~\ref{subsec:construction}, Proposition~\ref{prop:algebra}, and Proposition~\ref{prop:gamma}. Consider a group $H$ acting polarly with cohomogeneity two and section $\Sigma$ on $\bar{M}^2(c)$, and fix a regular point $p\in \Sigma$, and a unit vector $w\in T_p\Sigma$. Let $\gamma$ be one of the two locally defined unit speed curves on $\Sigma$ with $\gamma(0)=p$ and $\dot{\gamma}(0)=w$ such that $H\cdot\gamma$ has two principal curvatures (cf.~Section~\ref{sect:construction}). We have to show, under the assumption that $H\cdot\gamma$ is everywhere non-Hopf, that the curve $\gamma$ can be obtained by means of the initial value problem stated in Remark~\ref{remark:ODE}. We start by regarding $M=H\cdot\gamma$ as a real hypersurface with two principal curvatures in $\bar{M}^2(c)$, so that we can apply to it the study developed in sections~\ref{sect:setup} to~\ref{sect:distributions}. Let $\xi$ be a unit normal vector field along $M$. Recall from Section~\ref{sect:construction} that the principal curvatures $\alpha_0$ and $\beta_0$ of $H\cdot p$ at $p$ with respect to $\xi_p$ coincide with the two principal curvatures of $M$ at $p$. Assume without restriction of generality that $\beta_0$ has multiplicity two as principal curvature of $M$. In view of Proposition~\ref{prop:gamma}, it is enough to show that the tangent vector field to $\gamma$ is given by the vector field $A$ introduced in Section~\ref{sect:setup}. The Hopf vector $J\xi$ is orthogonal to $\dot{\gamma}$ along $\gamma$, since $\Sigma$ is totally real and $\xi_{\gamma(t)}$ is tangent to $\Sigma$ for all $t$. Moreover, as explained in Section~\ref{sect:construction}, $\dot{\gamma}(t)$~is a principal curvature vector of the hypersurface $M$ for all $t$. Therefore, it follows that $\dot{\gamma}$ is collinear with $A$. Since both $A$ and $\dot{\gamma}$ have unit length, reversing the sign of $\xi$ if necessary, we have that $\dot{\gamma}(t)=A_{\gamma(t)}$ for all $t$. This proves Remark~\ref{remark:ODE}.

\begin{remark}
Let us explain why, in the classification part of our Main Theorem, we make the assumption that the real hypersurface $M$ is non-Hopf at every point. On the one hand, if it were Hopf (at every point), arguments similar to those in sections~\ref{sect:setup} and~\ref{sect:Levi-Civita} show that $M$ would have constant principal curvatures, thus leading to well-known examples. On the other hand, we should consider the case when $M$ is Hopf only along a subset of $M$ whose complement in $M$ is open and dense. In this situation, the vector fields $V$ and $A$ in Proposition~\ref{prop:algebra} might not be well-defined in the points where $M$ is Hopf. However, in any case we know that there is an open and dense subset of $M\subset \bar{M}^2(c)$ that has the structure described in the Main Theorem.
\end{remark}


\end{document}